\documentclass{amsart}
\usepackage[english]{babel}
\usepackage{amssymb}
    \usepackage{mathtools}  
      \mathtoolsset{showonlyrefs}  

\theoremstyle{plain}
\newtheorem{definition}{Definition}
\theoremstyle{plain}
\newtheorem{lemma}{Lemma}
\theoremstyle{remark}
\newtheorem{remark}{Remark}
\theoremstyle{plain}
\newtheorem{theorem}{Theorem}

\begin{document}

\

\title{Evaluation of the mass of an asymptotically hyperbolic manifold}

\author{Xiaoxiang Chai}

 \email{chaixiaoxiang@gmail.com}
\begin{abstract}
  We show that the mass of an asymptotically hyperbolic manifold with a
  noncompact boundary can be evaluated via the Ricci tensor and the second
  fundamental form by using purely coordinates. The method is analog to
  Miao-Tam's approach to the asymptotically flat manifold.
\end{abstract}

{\maketitle}

\section{Introduction}

Let $\mathbb{H}^n$ be the hyperbolic $n$-space with constant sectional
curvature -1. Take an arbitrary point $o \in \mathbb{H}^n$ as the origin and
let $r (x) =\ensuremath{\operatorname{dist}}_{\mathbb{H}^n} (o, x)$ be the
geodesic distance from a point $x$ to the origin and $V = \cosh r$. The
rotationally symmetric form of the standard metric $b$ on $\mathbb{H}^n$ is
then written as $d r^2 + \sinh^2 r \sigma$ on $(0, \infty) \times
\mathbb{S}^{n - 1}$ where $\sigma$ is the standard metric on the unit $(n -
1)$-sphere.

As a model for the standard $n$-space $\mathbb{H}^n$ we may also take
$\mathbb{H}^{n - 1} \times \mathbb{R}$. Any point $x \in \mathbb{H}^n$ has the
coordinate $x = (x\prime, s)$. We assume that $o = (o\prime, 0) \in
\mathbb{H}^{n - 1} \times \mathbb{R}$ where we take $o\prime$ as the origin in
$\mathbb{H}^{n - 1}$. Let $U (x\prime) = \cosh
(\ensuremath{\operatorname{dist}}_{\mathbb{H}^{n - 1}} (o\prime, x\prime))$,
now the metric $b$ takes the form
\begin{equation}
  b := \bar{h} + U^2 d s \otimes d s
\end{equation}
where $\bar{h}$ is the metric for $\mathbb{H}^{n - 1}$. Note that $U = V$ when
$s = 0$. Now define
\begin{equation}
  \mathbb{H}^n_+ =\mathbb{H}^{n - 1} \times \{s \in \mathbb{R}: s \geqslant
  0\}
\end{equation}
and $B^n_+ = \{x \in \mathbb{H}^n_+ : \ensuremath{\operatorname{dist}}(x, o)
\leqslant 1\}$. Now let $\bar{\nabla}$ and $\bar{D}$ be respectively the
standard connection on $\mathbb{H}^n$ and $\mathbb{H}^{n - 1} =\mathbb{H}^{n -
1} \times \{0\}$. Denote the Christoffel symbols of $\mathbb{H}^{n - 1}$ by
$(\Gamma)$.

We let the Latin letters $i, j, k, l, \cdots$ range from 1 to $n$ and the
Greek letters $\alpha, \beta, \gamma \cdots$ range from 1 to $n - 1$. The
letter $n$ of course denotes the $s$ factor of $(x\prime, s) \in
\mathbb{H}^n$.

Motivated by the notion of an asymptotically flat manifold with a noncompact
boundary {\cite{almaraz-positive-2016}}, we formulate a similar notion in the
settings of asymptotically hyperbolic manifolds.

A Riemannian manifold $(M^n, g)$ is called {\itshape{asymptotically hyperbolic
with a noncompact boundary}} of decay order $\tau > \frac{n}{2}$ if there
exist a compact set $K$ and a diffeomorphism $\Psi : M\backslash K \to
\mathbb{H}^n_+ \backslash B_+^n$ such that $(\Psi^{- 1})^{\ast} g$ is
uniformly equivalent with $b$ and
\begin{equation}
  \|e\|_b + \| \bar{\nabla} e\|_b + \| \bar{\nabla} \bar{\nabla} e\|_b = O
  (e^{- \tau r}) \label{decay rate}
\end{equation}
where $e : = (\Psi^{- 1})^{\ast} g - b$. For the usual
{\itshape{asymptotically hyperbolic}} manifolds, one could just drop the
positive signs in $\mathbb{H}^n_+ \backslash B_+^n$.

\begin{definition}
  \label{mass} If $(M^n, g)$ is an asymptotically hyperbolic manifold with
  decay order $\tau > \frac{n}{2}$, $V (R + n (n - 1))$ is integrable and $V H
  \in L^1 (\partial M)$. Let $\{D_q \}_{q = 1}^{\infty}$ be a sequence of open
  sets with Lipschitz boundary $\partial D_q$. The boundary $\partial D_q$ are
  made of two portions, one is $\Sigma_q = \partial D_q \cap
  \ensuremath{\operatorname{int}}M$ and the other is $\Pi_q = \partial D_q
  \cap \partial M$. $\Pi_q$ and $\Sigma_q$ share the same boundary $S_q$. Let
  $r_q = \inf_{x \in \Sigma_q} \ensuremath{\operatorname{dist}}_{\mathbb{H}^n}
  (o, x)$.
  \begin{equation}
    \lim_{q \to \infty} r_q = \infty, \text{ } |S_q | \leqslant C \sinh^{n -
    2} r_q \text{ and } | \Sigma_l | \leqslant C \sinh^{n - 1} r_q,
  \end{equation}
  where $C$ is a constant independent of $q$. Then
  \begin{equation}
    m (g) = \lim_{r \to \infty} \left[ \int_{\Sigma_q} (V \bar{\nabla}_l e_{j
    k} - e_{j k} \bar{\nabla}_l V) P^{i j k l} \nu_i + \int_{S_q} e_{\alpha n}
    \theta^{\alpha} \right] \label{eq:mass}
  \end{equation}
  exists and is finite. $m (g)$ is called the mass.
\end{definition}

Let $G$ denote the Einstein tensor $\ensuremath{\operatorname{Rc}}-
\frac{1}{2} R g$, we define the {\itshape{modified Einstein tensor}}
\[ \tilde{G} := G - \frac{1}{2} (n - 1) (n - 2) g. \]
We can evaluate the mass $m (g)$ in terms of the Ricci tensor and the second
fundamental form.

\begin{theorem}
  \label{thm:via ricci} Assume that $(M, g)$ and $\{D_q \}$ are as Definition
  \ref{mass}, then
  \begin{equation}
    m (g) = - \frac{2}{n - 2} \left[ \int_{\Sigma_q} \tilde{G} (X, \nu) +
    \int_{S_q} (A - H h) (X, \theta) \right] + o (1) . \label{eq:via ricci}
  \end{equation}
\end{theorem}

\begin{remark}
  While preparing the article, the author learnt the formula {\eqref{eq:mass}}
  and {\eqref{eq:via ricci}} is also independently found by
  {\cite{almaraz-mass-2018}}, {\cite{de-lima-mass-2018}}. The geometric
  invariance of {\eqref{eq:mass}} is due to Almaraz-de Lima
  {\cite{almaraz-mass-2018}}.
\end{remark}

It is well know in the community of general relativity that the ADM mass of an
asymptotically flat manifold can be evaluating via the Ricci tensor. We refer
to the work of Miao and Tam {\cite{miao-evaluation-2016}} and the references
therein for the history. Miao and Tam used purely coordinates, and there is
also a work of Herzlich {\cite{herzlich-computing-2016}} who used instead a
coordinate-free approach.

The author {\cite{chai-aspects-2018}} in his PhD thesis had proved a similar
formula for the ADM type mass defined in {\cite{almaraz-positive-2016}} using
both approaches from Miao-Tam and Herzlich, see also {\cite{chai-two-2018}}
and {\cite{de-lima-mass-2018}}. In {\cite{de-lima-mass-2018}}, the formula
{\eqref{eq:via ricci}} is proved based on a method of Herzlich while here in
this article is based on purely coordinates.

This type of formula is used to prove the convergence of the Hawking mass and
Brown-York mass to the ADM mass in the asymptotically flat manifold case. See
the work of Miao, Tam and Xie {\cite{miao-quasi-local-2017}}. Given a hypersurface
$\Sigma^{n - 1}$ with boundary intersecting the ambient boundary orthogonally,
with the formula {\eqref{eq:via ricci}} and with similar calculations as in
{\cite{miao-quasi-local-2017}}, one can define a {\itshape{Hawking type mass
with boundary}},
\begin{equation}
  m_H (\Sigma) : = | \Sigma |^{\frac{1}{n - 1}} \left\{ \int_{\Sigma}
  [R_{\Sigma} + (n - 1) (n - 2) - \frac{n - 2}{n - 1} H_{\Sigma}^2] + 2
  \int_{\partial \Sigma} H_{\partial \Sigma} \right\} .
\end{equation}
In particular, when $\Sigma$ is of dimension 2, by the Gauss-Bonnet theorem,
one has (up to a constant)
\begin{equation}
  | \Sigma |^{1 / 2} \{8 \pi \chi (\Sigma) - \int_{\Sigma} (H^2 - 4)\}
  \label{hawking 2}
\end{equation}
resembling the classical Hawking mass {\cite{hawking-gravitational-1968}}.

\

{\bfseries{Acknowledgement.}} The author would like to extend gratitude to
Prof. Martin Man-chun Li and Prof. Luen Fai Tam for helpful discussions. The
author would like to thank Levi Lima for communicating the preprints
{\cite{almaraz-mass-2018}}, {\cite{de-lima-mass-2018}}.

\section{Proofs}

Let $\nabla$ be the Levi-Civita connection of the metric $g$, $h$ the metric
induced on $\partial M$, $D$ be the connection induced on $\partial M$ and
$\eta$ the outward normal to $\partial M$. We use the Einstein summation
convention. Also, we add a bar to denote the corresponding quantities of the
background metric $b$.

Let $P^{i j k l} = \frac{1}{2} (g^{i k} g^{j l} - g^{i l} g^{j k})$, assuming
the decay rate {\eqref{decay rate}}, we have an expansion at infinity,
\begin{equation}
  V (R + n (n - 1)) = 2 \bar{\nabla}_i ((V \bar{\nabla}_l e_{j k} - e_{j k}
  \bar{\nabla}_l V) P^{i j k l}) + O (e^{- 2 \tau r + r}) . \label{decay
  scalar}
\end{equation}
This expansion is used to define the {\itshape{mass functional}} $H_{\Phi}
(V)$ by Chrusciel and Herzlich {\cite{chrusciel-mass-2003}}. A mass is
introduced for an asymptotically flat manifold with a noncompact boundary by
Almaraz, Barbosa and de Lima {\cite{almaraz-positive-2016}}, where an
expansion of the mean curvature of the noncompact boundary is also used to
define an ADM mass. Similarly, in the settings of an asymptotically hyperbolic
manifold with a noncompact boundary,

\begin{lemma}
  Let $\mathcal{C}^i = (V \bar{\nabla}_l e_{j k} - e_{j k} \bar{\nabla}_l V)
  P^{i j k l}$, then
  \begin{equation}
    2 V H = - 2 \langle \mathcal{C}, \eta \rangle + \bar{h}^{\alpha \beta}
    \bar{D}_{\alpha} e_{\beta n} + O (e^{- 2 \tau r + r}) \label{decay mean}
  \end{equation}
  along $\partial M$. Here, $\bar{D}_{\alpha} e_{\beta n}$ is understood as
  the covariant derivative of the tensor field $e (\partial_n, \cdot)
  |_{\partial M}$.
\end{lemma}

\begin{proof}
  Since we are using this decomposition $b = \bar{h} + U^2 d s \otimes d s$ \
  of background metric $b$ on $M$ and the length of $\partial_s$ or
  $\partial_n$ is of the order $O (e^r)$, the decay of metric $g$ written in
  components then has to be handled carefully. We choose the coordinate $x$ on
  $M$ such that the vector fields $\partial_{\alpha}$ on the $\mathbb{H}^{n -
  1}$ factor have uniformly bounded length (from both below and above) and the
  $n$-th coordinate to be just $s$. Then we have the decay for the metric $g$,
  \begin{equation}
    g_{\alpha \beta} = b_{\alpha \beta} + O (e^{- \tau r}), g_{\alpha n} = O
    (e^{- \tau r + r}), g_{n n} = V^2 (1 + O (e^{- \tau r})) .
  \end{equation}
  The we use the adjugate matrix from linear algebra to find the decay for the
  inverse metric $g^{- 1}$,
  \begin{equation}
    g^{\alpha \beta} = b^{\alpha \beta} + O (e^{- \tau r}), g^{\alpha n} = O
    (e^{- \tau r - r}), g^{n n} = V^{- 2} (1 + O (e^{- \tau r})) .
  \end{equation}
  We readily have (see also {\cite{almaraz-positive-2016}})
  \begin{equation}
    \eta = - (g^{n n})^{- 1 / 2} g^{n i} \partial_i, \eta_n = - (g^{n n})^{- 1
    / 2}, \eta_{\alpha} = 0 \label{normals}
  \end{equation}
  and the mean curvature
  \begin{equation}
    H = - h^{\alpha \beta} \langle \eta, \nabla_{\alpha} \partial_{\beta}
    \rangle = h^{\alpha \beta} (g^{n n})^{- 1 / 2} \Gamma_{\alpha \beta}^n .
  \end{equation}
  Here $h^{\alpha \beta}$, the induced inverse metric on $\partial M$, is also
  the inverse of $g_{\alpha \beta}$ viewed as an $(n - 1) \times (n - 1)$
  matrix. Indeed, along $\partial M$, one can check easily from
  {\eqref{normals}} that $h^{\alpha \beta} g_{\beta \gamma} =
  \delta_{\gamma}^{\alpha}$, and the decay of $h^{\alpha \beta}$ is
  
  \begin{align}
    h^{\alpha \beta} & = g^{\alpha \beta} - \eta^{\alpha} \eta^{\beta} =
    b^{\alpha \beta} + O (e^{- \tau r}) = \bar{h}^{\alpha \beta} + O (e^{-
    \tau r}) .
  \end{align}
  
  The following line from {\cite[Lemma 3.1]{dahl-penrose-2013}} is an easy
  calculation and often used.
  \begin{equation}
    \bar{\Gamma}_{\alpha \beta}^{\gamma} = (\bar{\Gamma})_{\alpha
    \beta}^{\gamma}, \quad \bar{\Gamma}_{\alpha \beta}^n =
    \bar{\Gamma}^{\alpha}_{\beta n} = 0, \quad \bar{\Gamma}_{\alpha n}^n =
    V^{- 1} \bar{D}_{\alpha} V = V^{- 1} \bar{\nabla}_{\alpha} V.
  \end{equation}
  As is well known that the difference of two Christoffel symbols is a tensor,
  then
  \begin{align}
    \Gamma_{\alpha \beta}^n = & \Gamma_{\alpha \beta}^n - \bar{\Gamma}_{\alpha
    \beta}^n\\
    = & \frac{1}{2} g^{n l} (\bar{\nabla}_{\alpha} g_{\beta l} +
    \bar{\nabla}_{\beta} g_{\alpha l} - \bar{\nabla}_l g_{\alpha \beta})\\
    = & \frac{1}{2} g^{n l} (\bar{\nabla}_{\alpha} e_{\beta l} +
    \bar{\nabla}_{\beta} e_{\alpha l} - \bar{\nabla}_l e_{\alpha \beta})\\
    = & \frac{1}{2} g^{n \gamma} (\bar{\nabla}_{\alpha} e_{\beta \gamma} +
    \bar{\nabla}_{\beta} e_{\alpha \gamma} - \bar{\nabla}_{\gamma} e_{\alpha
    \beta}) + \frac{1}{2} g^{n n} (\bar{\nabla}_{\alpha} e_{\beta n} +
    \bar{\nabla}_{\beta} e_{a n} - \bar{\nabla}_n e_{\alpha \beta})\\
    = & O (e^{- 2 \tau r - r}) + \frac{1}{2} V^{- 2} (1 + O (e^{- \tau r}))
    (\bar{\nabla}_{\alpha} e_{\beta n} + \bar{\nabla}_{\beta} e_{a n} -
    \bar{\nabla}_n e_{\alpha \beta})\\
    = & \frac{1}{2} V^{- 2} (\bar{\nabla}_{\alpha} e_{\beta n} +
    \bar{\nabla}_{\beta} e_{\alpha n} - \bar{\nabla}_n e_{\alpha \beta}) + O
    (e^{- 2 \tau r - r})\\
    = & O (e^{- \tau r - r}) .
  \end{align}
Since $(g^{n n})^{- 1 / 2} = V (1 + O (e^{- \tau r}))$,
  \begin{equation}
    2 A_{\alpha \beta} = V^{- 1} (\bar{\nabla}_{\alpha} e_{\beta n} +
    \bar{\nabla}_{\beta} e_{\alpha n} - \bar{\nabla}_n e_{\alpha \beta}) + O
    (e^{- \tau r}) . \label{decay second fundamental form}
  \end{equation}
  Expanding $2 V H$ at infinity,
  \begin{align}
    2 V H & = 2 V h^{\alpha \beta} (g^{n n})^{- 1 / 2} \Gamma^n_{\alpha
    \beta}\\
    & = 2 V (b^{\alpha \beta} + O (e^{- \tau r})) V (1 + O (e^{- \tau r}))
    \Gamma_{\alpha \beta}^n\\
    & = b^{\alpha \beta} (2 \bar{\nabla}_{\alpha} e_{\beta n} -
    \bar{\nabla}_n e_{\alpha \beta}) + O (e^{- 2 \tau r + r}) . \label{decay
    mean curvature}
  \end{align}
  Expanding $2 \langle \mathcal{C}, \eta \rangle$ at infinity along $\partial
  M$,
  \begin{align}
    2 \langle C, \eta \rangle & = \eta_i (V \bar{\nabla}_l e_{j k} - e_{j k}
    \bar{\nabla}_l V) (g^{i k} g^{j l} - g^{i l} g^{j k})\\
    & = \bar{\eta}_i (V \bar{\nabla}_l e_{j k} - e_{j k} \bar{\nabla}_l V)
    (b^{i k} b^{j l} - b^{i l} b^{j k}) + O (e^{- 2 \tau r + r})\\
    & = - V (V \bar{\nabla}_l e_{j k} - e_{j k} \bar{\nabla}_l V) (b^{n k}
    b^{j l} - b^{n l} b^{j k}) + O (e^{- 2 \tau r + r})\\
    & = - (\bar{\nabla}_l e_{j n} - e_{j n} V^{- 1} \bar{\nabla}_l V) b^{j l}
    \\
    & \quad + (\bar{\nabla}_n e_{j k} - e_{j k} V^{- 1} \bar{\nabla}_n V)
    b^{j k} + O (e^{- 2 \tau r + r})\\
    & = - b^{j l} \bar{\nabla}_l e_{j n} + e_{\alpha n} b^{\alpha \beta} V^{-
    1} \bar{\nabla}_{\beta} V + b^{i j} \bar{\nabla}_n e_{i j} + O (e^{- 2
    \tau r + r}),
  \end{align}
  where in the second line we can obtain the expansion by returning
  temporarily to a coordinate whose coordinate vector fields have uniformly
  bounded length (from both below and above), and $\bar{\eta}$ is the outward
  normal to $\partial M$ under the metric $b$, and in the last line we have
  also used that
  \[ \bar{\nabla}_n V = \langle \partial_n, \bar{\nabla} V \rangle = \langle
     \partial_n, \bar{\nabla} r \rangle V\prime = 0. \]
  Finally, noting that $b^{\alpha \beta} = \bar{h}^{\alpha \beta}$ and
  $b_{\alpha \beta} = \bar{h}_{\alpha \beta}$, we have
  \begin{align}
    2 V H + 2 \langle C, \eta \rangle & = b^{\alpha \beta} (2
    \bar{\nabla}_{\alpha} e_{\beta n} - \bar{\nabla}_n e_{\alpha \beta}) + O
    (e^{- 2 \tau r + r})\\
    & \quad - b^{j l} \bar{\nabla}_l e_{j n} + e_{\alpha n} b^{\alpha \beta}
    V^{- 1} \bar{\nabla}_{\beta} V + b^{i j} \bar{\nabla}_n e_{i j}\\
    & = \bar{h}^{\alpha \beta} (\bar{\nabla}_{\alpha} e_{\beta n} + e_{\beta
    n} V^{- 1} \bar{\nabla}_{\alpha} V) + O (e^{- 2 \tau r + r}) .
  \end{align}
  Since we view $e (\partial_n, \cdot) |_{\partial M}$ as a tensor field on
  $\partial M$, inserting the following relation to the above
  \begin{align}
    \bar{\nabla}_{\alpha} e_{\beta n} & = \partial_{\alpha} e_{\beta n} -
    \bar{\Gamma}^i_{\alpha \beta} e_{i n} - \bar{\Gamma}^i_{\alpha n} e_{i
    \beta}\\
    & = (\partial_{\alpha} e_{\beta n} - \bar{\Gamma}^{\gamma}_{\alpha \beta}
    e_{\gamma n}) - \bar{\Gamma}^n_{\alpha \beta} e_{n n} -
    \bar{\Gamma}^{\gamma}_{\alpha n} e_{\gamma \beta} - \bar{\Gamma}^n_{\alpha
    n} e_{n \beta}\\
    & = \bar{D}_{\alpha} e_{\beta n} - \bar{\Gamma}^n_{\alpha n} e_{n
    \beta}\\
    & = \bar{D}_{\alpha} e_{\beta n} - e_{\beta n} V^{- 1} \bar{D}_{\alpha}
    V\\
    & = \bar{D}_{\alpha} e_{\beta n} - e_{\beta n} V^{- 1}
    \bar{\nabla}_{\alpha} V, \label{eq:boundary derivative}
  \end{align}
  finishes the proof.
\end{proof}
We establish the following decay of $\tilde{G}$. The proof is also used later
in the paper.
\begin{lemma}
  For an asymptotically hyperbolic manifold $(M, g)$ with decay rate $\tau > n
  / 2$, the modified Einstein tensor expands at infinity as
  \begin{align}
    - \tilde{G}_{i k} & = 2 (1 - n) e_{i k} + (\bar{\nabla}_i \bar{\nabla}_k E
    - \bar{\nabla}^l \bar{\nabla}_i e_{k l} - \bar{\nabla}^l \bar{\nabla}_k
    e_{i l} + \bar{\nabla}^l \bar{\nabla}_l e_{i k})\\
    & \quad - (1 - n) b_{i k} E - b_{i k} (\bar{\Delta} E - \bar{\nabla}^j
    \bar{\nabla}^l e_{j l}) + O (e^{- 2 \tau r}) .
  \end{align}
\end{lemma}

\begin{proof}
  Let $\Lambda = \Gamma - \bar{\Gamma}$, $\Lambda$ is the difference of the
  Christoffel symbols and hence a tensor. More specifically,
  \begin{align}
    & \frac{1}{2} g^{k l} (\bar{\nabla}_i e_{j l} + \bar{\nabla}_j e_{i l} -
    \bar{\nabla}_l e_{i j})\\
    = & \frac{1}{2} g^{k l} (\bar{\nabla}_i g_{j l} + \bar{\nabla}_j g_{i l} -
    \bar{\nabla}_l g_{i j})\\
    = & \frac{1}{2} g^{k l} (\partial_i g_{j l} - \bar{\Gamma}^s_{i l} g_{s j}
    - \bar{\Gamma}_{i j}^s g_{s l}) + \frac{1}{2} g^{k l} (\partial_j g_{i l}
    - \bar{\Gamma}^s_{j l} g_{s i} - \bar{\Gamma}^s_{i j} g_{s l})\\
    & \quad - \frac{1}{2} g^{k l} (\partial_l g_{i j} - \bar{\Gamma}^s_{j l}
    g_{s i} - \bar{\Gamma}_{i l}^s g_{s j})\\
    & = \Gamma_{i j}^k - \bar{\Gamma}_{i j}^k = \Lambda_{i j}^k = O (e^{-
    \tau r}) .
  \end{align}
  Expressing $R_{i j k}^{\phantom{i j k} l}$ in terms of $\bar{R}_{i j
  k}^{\phantom{i j k} l}$ and $\Lambda$, we have
  \begin{align}
    R_{i j k}^{\phantom{i j k} l} & = \partial_i \Gamma_{j k}^l - \partial_j
    \Gamma_{i k}^l + \Gamma_{j k}^m \Gamma_{i m}^l - \Gamma_{i k}^m \Gamma_{j
    m}^l\\
    & = \partial_i (\Lambda_{j k}^l + \bar{\Gamma}_{j k}^l) - \partial_j
    (\Lambda_{i k}^l + \bar{\Gamma}_{i k}^l)\\
    & \quad + (\Lambda_{j k}^m + \bar{\Gamma}_{j k}^m) (\Lambda_{i m}^l +
    \bar{\Gamma}_{i m}^l) - (\Lambda_{i k}^m + \bar{\Gamma}_{i k}^m)
    (\Lambda_{j m}^l + \bar{\Gamma}_{j m}^l)\\
    & = \partial_i \Lambda_{j k}^l - \Lambda_{s k}^l \bar{\Gamma}_{i j}^s -
    \Lambda_{s j}^l \bar{\Gamma}^s_{i k} + \Lambda_{j k}^s \bar{\Gamma}_{i
    s}^l + \bar{R}_{i j k}^{\phantom{i j k} l}\\
    & \quad - \partial_j \Lambda_{i k}^l + \Lambda_{i s}^l \bar{\Gamma}_{j
    k}^s + \Lambda_{s k}^l \bar{\Gamma}_{i j}^s - \Lambda_{i k}^s
    \bar{\Gamma}_{j s}^l + \Lambda_{j k}^m \Lambda_{i m}^l - \Lambda_{i k}^m
    \Lambda_{j m}^l\\
    & = \bar{R}_{i j k}^{\phantom{i j k} l} + \bar{\nabla}_i \Lambda_{j k}^l
    - \bar{\nabla}_j \Lambda_{i k}^l + \Lambda_{j k}^m \Lambda_{i m}^l -
    \Lambda_{i k}^m \Lambda_{m j}^l .
  \end{align}
  In short, the Riemann curvature tensor has the decay,
  \begin{align}
    R_{i j k}^{\phantom{i j k} l} & = \bar{R}_{i j k}^{\phantom{i j k} l} +
    \bar{\nabla}_i \Lambda_{j k}^l - \bar{\nabla}_j \Lambda_{i k}^l + O (e^{-
    2 \tau r}) .
  \end{align}
  We also readily find the decay of $\bar{\nabla}_i g^{j k}$,
  \begin{align}
    \bar{\nabla}_i g^{j k} & = \partial_i g^{j k} + \bar{\Gamma}_{i l}^k g^{j
    l} + \bar{\Gamma}_{i l}^j g^{k l}\\
    & = (\partial_i g^{j k} + \Gamma_{i l}^k g^{j l} + \Gamma_{i l}^j g^{k
    l}) - (\Lambda_{i l}^k g^{j l} + \Lambda_{i l}^j g^{k l})\\
    & = - (\Lambda_{i l}^k g^{j l} + \Lambda_{i l}^j g^{k l}) = O (e^{- \tau
    r}) .
  \end{align}
  Since $g = b + e$, the inverse metric $g^{- 1}$ by the formula of invertible
  matrices from elementary linear algebra, $g^{i j} = b^{i j} + O (e^{- \tau
  r})$, hence the decay of $\bar{\nabla}_i \Lambda_{j k}^l$ is
  \begin{align}
    2 \bar{\nabla}_i \Lambda_{j k}^l & = \bar{\nabla}_i (g^{l s}
    (\bar{\nabla}_j e_{k s} + \bar{\nabla}_k e_{j s} - \bar{\nabla}_s e_{j
    k})^{})\\
    & = b^{s l} (\bar{\nabla}_j e_{k s} + \bar{\nabla}_k e_{j s} -
    \bar{\nabla}_s e_{j k}) + O (e^{- 2 \tau r}) .
  \end{align}
  We use the shorthand $E := \ensuremath{\operatorname{tr}}_b e$, the decay of
  the Ricci tensor is
  \begin{align}
    - 2 R_{i k} & = 2 R_{i j k}^{\phantom{i j k} j}\\
    & = - 2 \bar{R}_{i k} + 2 \bar{\nabla}_i \Lambda_{j k}^j - 2
    \bar{\nabla}_j \Lambda_{i k}^j + O (e^{- 2 \tau r})\\
    & = - 2 \bar{R}_{i k} + b^{j l} \bar{\nabla}_i (\bar{\nabla}_j e_{l k} +
    \bar{\nabla}_k e_{l j} - \bar{\nabla}_l e_{j k})\\
    & \quad - b^{j l} \bar{\nabla}_j (\bar{\nabla}_i e_{k l} + \bar{\nabla}_k
    e_{i l} - \bar{\nabla}_l e_{i k}) + O (e^{- 2 \tau r})\\
    & = - 2 (1 - n) b_{i k}\\
    & \quad + (\bar{\nabla}_i \bar{\nabla}_k E_{} - \bar{\nabla}^l
    \bar{\nabla}_i e_{k l} - \bar{\nabla}^l \bar{\nabla}_k e_{i l} +
    \bar{\nabla}^l \bar{\nabla}_l e_{i k}) + O (e^{- 2 \tau r}) .
  \end{align}

  There are two metrics involved, so we should raise and lower the indices
  explicitly. Also, the scalar curvature $R$ decays as
  \begin{align}
    - 2 R & = - 2 g^{i k} R_{i k}\\
    & = - 2 g^{i k} b_{i k} (1 - n)\\
    & \quad + b^{i k} (\bar{\nabla}_i \bar{\nabla}_k E - \bar{\nabla}^l
    \bar{\nabla}_i e_{k l} - \bar{\nabla}^l \bar{\nabla}_k e_{i l} +
    \bar{\nabla}^l \bar{\nabla}_l e_{i k}) + O (e^{- 2 \tau r})\\
    & = - 2 g^{i k} b_{i k} (1 - n) + 2 \bar{\Delta} E - 2 \bar{\nabla}^k
    \bar{\nabla}^l e_{k l} + O (e^{- 2 \tau r}) .
  \end{align}
  By a simple Taylor expansion argument, one has the elementary fact from
  linear algebra: Let $A, B$ be two symmetric $n \times n$ matrices, $A$ is
  positive definite and all entries of $B$ are sufficiently small such that $C
  := (A + B)^{- 1}$ exists, then
  \begin{equation}
    \sum_{i, j = 1}^n C_{i j} A_{i j} = n - \sum_{i, j = 1}^n A_{i j} B_{i j}
    + O (|B|^2) .
  \end{equation}
  So
  \begin{equation}
    R = (n - E) (1 - n) - (\bar{\Delta} E - \bar{\nabla}^k \bar{\nabla}^l e_{k
    l}) + O (e^{- 2 \tau r}) .
  \end{equation}
  
  Finally, we have the decay of the modified Einstein tensor $\tilde{G}$,
  \begin{align}
    - 2 \tilde{G}_{i k} & = - 2 R_{i k} + R g_{i k} + (n - 1) (n - 2) g_{i
    k}\\
    & = - 2 (1 - n) b_{i k} + (\bar{\nabla}_i \bar{\nabla}_k E -
    \bar{\nabla}^l \bar{\nabla}_i e_{k l} - \bar{\nabla}^l \bar{\nabla}_k e_{i
    l} + \bar{\nabla}^l \bar{\nabla}_l e_{i k})\\
    & \quad + (n - E) (1 - n) g_{i k} - g_{i k} (\bar{\Delta} E -
    \bar{\nabla}^j \bar{\nabla}^l e_{j l})\\
    & \quad + (n - 1) (n - 2) g_{i k} + O (e^{- 2 \tau r})\\
    & = 2 (1 - n) e_{i k} + (\bar{\nabla}_i \bar{\nabla}_k E - \bar{\nabla}^l
    \bar{\nabla}_i e_{k l} - \bar{\nabla}^l \bar{\nabla}_k e_{i l} +
    \bar{\nabla}^l \bar{\nabla}_l e_{i k})\\
    & \quad - (1 - n) b_{i k} E - b_{i k} (\bar{\Delta} E - \bar{\nabla}^j
    \bar{\nabla}^l e_{j l}) + O (e^{- 2 \tau r}) \label{decay einstein}
  \end{align}
   concluding the proof.
\end{proof}
Now we turn to the proof of Theorem \ref{thm:via ricci}.
\begin{proof}[Proof of Theorem \ref{thm:via ricci}]
  For brevity, we suppress the index $q$ in $D_q$ and similar objects.
  
  {\bfseries{Step 1. Term involving the modified Einstein tensor
  $\tilde{G}$.}} Set
  \begin{equation}
    I := \int_{\Sigma} \bar{\eta}^k \bar{\nabla}^i V [\bar{\nabla}_i
    \bar{\nabla}_k E - \bar{\nabla}^l \bar{\nabla}_i e_{k l} - \bar{\nabla}^l
    \bar{\nabla}_k e_{i l} + \bar{\nabla}^l \bar{\nabla}_l e_{i k}],
  \end{equation}
  and
  \begin{equation}
    J = \int_{\Pi} \bar{\eta}^k \bar{\nabla}^i V [\bar{\nabla}_i
    \bar{\nabla}_k E - \bar{\nabla}^l \bar{\nabla}_i e_{k l} - \bar{\nabla}^l
    \bar{\nabla}_k e_{i l} + \bar{\nabla}^l \bar{\nabla}_l e_{i k}] .
  \end{equation}
  Note that the integrands of $I$ and $J$ comes from $\tilde{G}$.
  \begin{align}
    I = & \int_{\Sigma} \bar{\eta}^k \bar{\nabla}^i V [\bar{\nabla}_i
    \bar{\nabla}_k E - \bar{\nabla}^l \bar{\nabla}_i e_{k l} - \bar{\nabla}^l
    \bar{\nabla}_k e_{i l} + \bar{\nabla}^l \bar{\nabla}_l e_{i k}]\\
    = & \int_{\partial D} - \int_{\Pi}\\
    = & \int_D \bar{\nabla}^k \bar{\nabla}^i V [\bar{\nabla}_i \bar{\nabla}_k
    E - \bar{\nabla}^l \bar{\nabla}_i e_{k l} - \bar{\nabla}^l \bar{\nabla}_k
    e_{i l} + \bar{\nabla}^l \bar{\nabla}_l e_{i k}]\\
    & \quad + \int_D \bar{\nabla}^i V \bar{\nabla}^k [\bar{\nabla}_i
    \bar{\nabla}_k E - \bar{\nabla}^l \bar{\nabla}_i e_{k l} - \bar{\nabla}^l
    \bar{\nabla}_k e_{i l} + \bar{\nabla}^l \bar{\nabla}_l e_{i k}] - J\\
    = & \int_D V b^{i k} [\bar{\nabla}_i \bar{\nabla}_k E - \bar{\nabla}^l
    \bar{\nabla}_i e_{k l} - \bar{\nabla}^l \bar{\nabla}_k e_{i l} +
    \bar{\nabla}^l \bar{\nabla}_l e_{i k}] - J\\
    & \quad + \int_D \bar{\nabla}^i V [\bar{\nabla}^k \bar{\nabla}_i
    \bar{\nabla}_k E - \bar{\nabla}^k \bar{\nabla}^l \bar{\nabla}_i e_{k l}]\\
    & \quad + \int_D \bar{\nabla}^i V [- \bar{\nabla}^k \bar{\nabla}^l
    \bar{\nabla}_k e_{i l} + \bar{\nabla}^k \bar{\nabla}^l \bar{\nabla}_l e_{i
    k}]\\
    = & \int_D 2 V [\bar{\nabla}^k \bar{\nabla}_k E - \bar{\nabla}^k
    \bar{\nabla}^l e_{k l}] + I_1 + I_2 - J,
  \end{align}
  where we denote
  \begin{equation}
    I_1 := \int_D \bar{\nabla}^i V [\bar{\nabla}^k \bar{\nabla}_i
    \bar{\nabla}_k E - \bar{\nabla}^k \bar{\nabla}^l \bar{\nabla}_i e_{k l}]
    \label{I1}
  \end{equation}
  and
  \begin{equation}
    I_2 := \int_D \bar{\nabla}^i V [- \bar{\nabla}^k \bar{\nabla}^l
    \bar{\nabla}_k e_{i l} + \bar{\nabla}^k \bar{\nabla}^l \bar{\nabla}_l e_{i
    k}] . \label{I2}
  \end{equation}
  One of the main steps is the calculation of the term $I_1$. First,
  exchanging the order of indices leads to
  \begin{equation}
    \bar{\nabla}^k \bar{\nabla}_i \bar{\nabla}_k E = \bar{\nabla}_i
    \bar{\nabla}^k \bar{\nabla}_k E - \bar{R}^{k \phantom{i j} l}_{\phantom{i}
    i k} \bar{\nabla}_l E = \bar{\nabla}_i \bar{\nabla}^k \bar{\nabla}_k E +
    (1 - n) \bar{\nabla}_i E,
  \end{equation}
  and similarly,
  \begin{align}
    \bar{\nabla}^k \bar{\nabla}^l \bar{\nabla}_i e_{k l} & = \bar{\nabla}^k
    (\bar{\nabla}_i \bar{\nabla}^l e_{k l} - \bar{R}^{l \phantom{i j}
    j}_{\phantom{i} i k} e_{j l} - \bar{R}^{l \phantom{i j} j}_{\phantom{i} i
    l} e_{j k})\\
    & = (\bar{\nabla}_i \bar{\nabla}^k \bar{\nabla}^l e_{k l} - \bar{R}^{k
    \phantom{i j} j}_{\phantom{i} i k} \bar{\nabla}^l e_{j l} - \bar{R}^{k
    \phantom{i j} j}_{\phantom{i} i l} \bar{\nabla}^l e_{j k} - \bar{R}^{k
    \phantom{i} l}_{\phantom{i} i \phantom{i} \phantom{i} j} \bar{\nabla}^j
    e_{k l})\\
    & \quad - \bar{R}^{l \phantom{i j} j}_{\phantom{i} i k} \bar{\nabla}^k
    e_{j l} - \bar{R}^{l \phantom{i j} j}_{\phantom{i} i l} \bar{\nabla}^k
    e_{j k}\\
    & = \bar{\nabla}_i \bar{\nabla}^k \bar{\nabla}^l e_{k l} - (2 n - 1)
    \bar{\nabla}^l e_{i l} + \bar{\nabla}_i E.
  \end{align}
  Integration by parts gives
  \begin{align}
    I_1 & = \int_D \bar{\nabla}^i V [\bar{\nabla}^k \bar{\nabla}_i
    \bar{\nabla}_k E - \bar{\nabla}^k \bar{\nabla}^l \bar{\nabla}_i e_{k l}]\\
    = & \int_D \bar{\nabla}^i V \bar{\nabla}_i (\bar{\nabla}^k \bar{\nabla}_k
    E - \bar{\nabla}^k \bar{\nabla}^l e_{k l})\\
    & \quad + \int_D [(2 n - 1) \bar{\nabla}^i V \bar{\nabla}^k e_{i k} - n
    \bar{\nabla}^i V \bar{\nabla}_i E]\\
    = & \int_{\partial D} \bar{\eta}_i \overline{\nabla^{}}^i V
    (\bar{\nabla}^k \bar{\nabla}_k E - \bar{\nabla}^k \bar{\nabla}^l e_{k l})
    - \int_D \bar{\nabla}_i \bar{\nabla}^i V (\bar{\nabla}^k \bar{\nabla}_k E
    - \bar{\nabla}^k \bar{\nabla}^l e_{k l})\\
    & \quad + \int_D [(2 n - 1) \bar{\nabla}^i V \bar{\nabla}^k e_{i k} - n
    \bar{\nabla}^i V \bar{\nabla}_i E] .
  \end{align}
  For the integrand of the term $I_2$,
  \begin{align}
    & - \bar{\nabla}^k \bar{\nabla}^l \bar{\nabla}_k e_{i l} + \bar{\nabla}^l
    \bar{\nabla}^k \bar{\nabla}_k e_{i l}\\
    = & \bar{R}^{k l \phantom{i} j}_{\phantom{i j} k} \bar{\nabla}_j e_{i l} +
    \bar{R}^{k l \phantom{i} j}_{\phantom{i j} l} \bar{\nabla}_k e_{i j} +
    \bar{R}^{k l \phantom{i} j}_{\phantom{i j} i} \bar{\nabla}_k e_{j l}\\
    = & \bar{R}^{k l \phantom{i} j}_{\phantom{i j} i} \bar{\nabla}_k e_{j l}\\
    = & \bar{\nabla}_i E - \bar{\nabla}^j e_{j i},
  \end{align}
  so
  \begin{equation}
    I_2 = \int_D \bar{\nabla}^i V (\bar{\nabla}_i E - \bar{\nabla}^j e_{j i})
    .
  \end{equation}
  Since $\bar{\nabla}^i \bar{\nabla}_j V = V \delta^i_j$, $\bar{\Delta} V = n
  V$, $\bar{\eta}$ is orthogonal to $\bar{\nabla} V$ along $\partial M$, and a
  further integration by parts gives
  \begin{align}
    I & = (2 - n) \int_D V [\bar{\nabla}^k \bar{\nabla}_k E - \bar{\nabla}^k
    \bar{\nabla}^l e_{k l}]\\
    & \quad + \int_{\Sigma} \bar{\eta}_i \overline{\nabla^{}}^i V
    (\bar{\nabla}^k \bar{\nabla}_k E - \bar{\nabla}^k \bar{\nabla}^l e_{k
    l})\\
    & \quad + \int_D [(2 n - 2) \bar{\nabla}^j V \bar{\nabla}^k e_{j k} - (n
    - 1) \bar{\nabla}^l V \bar{\nabla}_l E] - J\\
    & = (2 - n) \int_{\partial D} \bar{\eta}^k V (\bar{\nabla}_k E -
    \bar{\nabla}^l e_{k l}) - J\\
    & \quad + \int_D n \bar{\nabla}^j V \bar{\nabla}^k e_{j k} - \int_D
    \bar{\nabla}^l V \bar{\nabla}_l E\\
    & \quad + \int_{\Sigma} \bar{\eta}_i \overline{\nabla^{}}^i V
    (\bar{\nabla}^k \bar{\nabla}_k E - \bar{\nabla}^k \bar{\nabla}^l e_{k l})
    .
  \end{align}
  Combining the above with {\eqref{decay einstein}}, and again from
  integration by parts,
  \begin{align}
    & - 2 \int_{\Sigma} \tilde{G}_{i k} \eta^k \bar{\nabla}^i V + o (1)\\
    = & - 2 \int_{\Sigma} \tilde{G}_{i k} \bar{\eta}^k \bar{\nabla}^i V\\
    = & \int_{\Sigma} 2 (1 - n) e_{i k} \bar{\eta}^k \bar{\nabla}^i V + (2 -
    n) \int_{\partial D} \bar{\eta}^k V (\bar{\nabla}_k E - \bar{\nabla}^l
    e_{k l})\\
    & \quad + \int_D (n \bar{\nabla}^j V \bar{\nabla}^k e_{j k} -
    \bar{\nabla}^l V \bar{\nabla}_l E) - J.\\
    & \quad - \int_{\Sigma} (1 - n) E \bar{\eta}_k \bar{\nabla}^k V\\
    = & (2 - n) \int_{\Sigma} [e_{i k} \bar{\eta}^k \bar{\nabla}^i V - E
    \bar{\eta}_k \bar{\nabla}^k V + \bar{\eta}^k V (\bar{\nabla}_k E -
    \bar{\nabla}^l e_{k l})]\\
    & \quad + (2 - n) \int_{\Pi} \bar{\eta}^k V (\bar{\nabla}_k E -
    \bar{\nabla}^l e_{k l}) - J\\
    & \quad + n \int_{\Pi} e_{j k} \bar{\eta}^k \bar{\nabla}^j V - \int_{\Pi}
    E \bar{\eta}_l \bar{\nabla}^l V.
  \end{align}
  Writing down the terms that do not appear in the mass expression
  {\eqref{eq:mass}} and \ noting that $\bar{D} V = \bar{\nabla} V$ and
  $\bar{\eta} \bot \bar{\nabla} V$ along $\partial M$, we have
  \begin{align}
    K := & (2 - n) \int_{\Pi} \bar{\eta}^k V (\bar{\nabla}_k E -
    \bar{\nabla}^l e_{k l}) + n \int_{\Pi} e_{j k} \bar{\eta}^k \bar{\nabla}^j
    V - \int_{\Pi} E \bar{\eta}_l \bar{\nabla}^l V\\
    & \quad - \int_{\Pi} \bar{\eta}^k \bar{\nabla}^i V [\bar{\nabla}_i
    \bar{\nabla}_k E - \bar{\nabla}^l \bar{\nabla}_i e_{k l} - \bar{\nabla}^l
    \bar{\nabla}_k e_{i l} + \bar{\nabla}^l \bar{\nabla}_l e_{i k}]\\
    = & - (2 - n) \int_{\Pi} (\bar{\nabla}_n E - \bar{\nabla}^l e_{n l}) - n
    \int_{\Pi} e_{\alpha n} V^{- 1} \bar{\nabla}^{\alpha} V\\
    & \quad + \int_{\Pi} V^{- 1} \bar{\nabla}^{\alpha} V
    [\bar{\nabla}_{\alpha} \bar{\nabla}_n E - \bar{\nabla}^l
    \bar{\nabla}_{\alpha} e_{n l} - \bar{\nabla}^l \bar{\nabla}_n e_{\alpha l}
    + \bar{\nabla}^l \bar{\nabla}_l e_{\alpha n}]
  \end{align}
  
  {\bfseries{Step 2. Term involving second fundamental form $A_{\alpha
  \beta}$.}} Similar to {\eqref{eq:boundary derivative}}, viewing
  $\bar{\nabla}_{\alpha} e_{\beta n}$ as a 2-tensor on $\partial M$, we have
  \begin{equation}
    \bar{D}_{\gamma} \bar{\nabla}_{\alpha} e_{\beta n} = \bar{\nabla}_{\gamma}
    \bar{\nabla}_{\alpha} e_{\beta n} + V^{- 1} \bar{\nabla}_{\gamma} V
    \bar{\nabla}_{\alpha} e_{\beta n} .
  \end{equation}
  We calculate the terms of type $\bar{D}_{\gamma} (V^{- 1} \bar{D}^{\alpha} V
  \bar{\nabla}_{\mu} e_{\beta n})$ as follows
  \begin{align}
    & \bar{D}_{\gamma} (V^{- 1} \bar{D}_{\alpha} V \bar{\nabla}_{\mu}
    e_{\beta n})\\
    = & - V^{- 2} \bar{D}_{\gamma} V \bar{D}_{\alpha} V \bar{\nabla}_{\mu}
    e_{\beta n} + V^{- 1} \bar{D}_{\alpha} V \bar{D}_{\gamma}
    \bar{\nabla}_{\mu} e_{\beta n} + V^{- 1} \bar{D}_{\gamma} \bar{D}_{\alpha}
    V \bar{\nabla}_{\mu} e_{\beta n}\\
    = & - V^{- 2} \bar{D}_{\gamma} V \bar{D}_{\alpha} V \bar{\nabla}_{\mu}
    e_{\beta n} + V^{- 1} \bar{D}_{\alpha} V (\bar{\nabla}_{\gamma}
    \bar{\nabla}_{\mu} e_{\beta n} + V^{- 1} \bar{\nabla}_{\gamma} V
    \bar{\nabla}_{\mu} e_{\beta n})\\
    & \quad + V^{- 1} \bar{D}_{\gamma} \bar{D}_{\alpha} V \bar{\nabla}_{\mu}
    e_{\beta n}\\
    = & V^{- 1} \bar{D}_{\alpha} V \bar{\nabla}_{\gamma} \bar{\nabla}_{\mu}
    e_{\beta n} + V^{- 1} \bar{D}_{\gamma} \bar{D}_{\alpha} V
    \bar{\nabla}_{\mu} e_{\beta n}\\
    = & V^{} \bar{D}_{\alpha} V \bar{\nabla}_{\gamma} \bar{\nabla}_{\mu}
    e_{\beta n} + b_{\gamma \alpha} \bar{\nabla}_{\mu} e_{\beta n},
  \end{align}
  where in the last line we have used the relation $\bar{D}_{\alpha}
  \bar{D}_{\beta} V = V b_{\alpha \beta}$, and it is important to keep in mind
  the range of the Greek indices.
  Switching the roles of the indices $\mu$ and $n$, one obtains a similar
  formula. 
  
  By recalling from the expansion {\eqref{decay second fundamental
  form}} and {\eqref{decay mean curvature}},
  \begin{align}
    & - 2 \int_{\partial \Pi} (A - H h) (\bar{D} V, \theta) + o (1)\\
    = & - 2 \int_{\partial \Pi} (A - H h) (\bar{D} V, \bar{\theta})\\
    = & \int_{\partial \Pi} \bar{D}^{\gamma} V \bar{\theta}_{\gamma} V^{- 1}
    b^{\alpha \beta} (2 \bar{\nabla}_{\alpha} e_{\beta n} - \bar{\nabla}_n
    e_{\alpha \beta})\\
    & \quad - \bar{D}^{\alpha} V \bar{\theta}^{\beta} V^{- 1}
    (\bar{\nabla}_{\alpha} e_{\beta n} + \bar{\nabla}_{\beta} e_{\alpha n} -
    \bar{\nabla}_n e_{\alpha \beta})\\
    = & \int_{\Pi} \bar{D}_{\gamma} (V^{- 1} \bar{D}^{\gamma} V b^{\alpha
    \beta} (2 \bar{\nabla}_{\alpha} e_{\beta n} - \bar{\nabla}_n e_{\alpha
    \beta}))\\
    & \quad - \int_{\Pi} \bar{D}^{\beta} (V^{- 1} \bar{D}^{\alpha} V
    (\bar{\nabla}_{\alpha} e_{\beta n} + \bar{\nabla}_{\beta} e_{\alpha n} -
    \bar{\nabla}_n e_{\alpha \beta}))\\
    = & (n - 2) \int_{\Pi} b^{\alpha \beta} (2 \bar{\nabla}_{\alpha} e_{\beta
    n} - \bar{\nabla}_n e_{\alpha \beta})\\
    & \quad + \int_{\Pi} V^{- 1} \bar{D}^{\gamma} V b^{\alpha \beta} (2
    \bar{\nabla}_{\gamma} \bar{\nabla}_{\alpha} e_{\beta n} -
    \bar{\nabla}_{\gamma} \bar{\nabla}_n e_{\alpha \beta})\\
    & \quad - \int_{\Pi} V^{- 1} \bar{\nabla}^{\alpha} V
    (\bar{\nabla}^{\beta} \bar{\nabla}_{\alpha} e_{\beta n} +
    \bar{\nabla}^{\beta} \bar{\nabla}_{\beta} e_{\alpha n} -
    \bar{\nabla}^{\beta} \bar{\nabla}_n e_{\alpha \beta})\\
    = : & L.
  \end{align}
  
  To finish the proof, it remains to prove that
  \begin{equation}
    K + L = (n - 2) \int_{\Pi} \bar{h}^{\alpha \beta} \bar{D}_{\alpha}
    e_{\beta n} .
  \end{equation}
  We collect terms involving second covariant derivatives of $e = g - b$ from
  $K + L$, and also note that there is a common factor $V^{- 1}
  \bar{D}^{\alpha} V$ shared by these terms after renaming. Not writing down
  the common factor and using that the $b = \bar{h} + V^2 d s \otimes d s$
  along $\partial M$,
  \begin{align}
    & [\bar{\nabla}_{\alpha} \bar{\nabla}_n E - \bar{\nabla}^l
    \bar{\nabla}_{\alpha} e_{n l} - \bar{\nabla}^l \bar{\nabla}_n e_{\alpha l}
    + \bar{\nabla}^l \bar{\nabla}_l e_{\alpha n}]\\
    & \quad - (\bar{\nabla}^{\beta} \bar{\nabla}_{\alpha} e_{\beta n} +
    \bar{\nabla}^{\beta} \bar{\nabla}_{\beta} e_{\alpha n} -
    \bar{\nabla}^{\beta} \bar{\nabla}_n e_{\alpha \beta})\\
    & \quad + b^{\gamma \beta} (2 \bar{\nabla}_{\alpha} \bar{\nabla}_{\gamma}
    e_{\beta n} - \bar{\nabla}_{\alpha} \bar{\nabla}_n e_{\gamma \beta})\\
    = & 2 (\bar{\nabla}_a \bar{\nabla}^{\beta} e_{\beta n} -
    \bar{\nabla}^{\beta} \bar{\nabla}_{\alpha} e_{\beta n}) +
    (\bar{\nabla}_{\alpha} \bar{\nabla}^n e_{n n} - \bar{\nabla}^n
    \bar{\nabla}_{\alpha} e_{n n})\\
    = & - 2 (\bar{R}_{\alpha \phantom{i} \beta}^{\phantom{i j} \beta
    \phantom{i j} l} e_{n l} - \bar{R}_{\alpha \phantom{i} n}^{\phantom{i j}
    \beta \phantom{i j} l} e_{\beta l}) - 2 \bar{R}_{\alpha \phantom{i j}
    n}^{\phantom{i j} n \phantom{i j} l} e_{n l}\\
    = & - 2 \bar{R}_{\alpha \phantom{i j} k}^{\phantom{i j} k \phantom{i j} l}
    e_{n l} = - 2 (1 - n) e_{n \alpha} .
  \end{align}
  Adding the common factor $V^{- 1} \bar{\nabla}^{\alpha} V$ back, we have
  \begin{align}
    & K + L\\
    = & (n - 2) \int_{\Pi} b^{\alpha \beta} \bar{\nabla}_{\alpha} e_{\beta n}
    - n \int_{\Pi} e_{\alpha n} V^{- 1} \bar{\nabla}^{\alpha} V\\
    & \quad - 2 (1 - n) \int_{\Pi} e_{\alpha n} V^{- 1} \bar{\nabla}^{\alpha}
    V\\
    = & (n - 2) \int_{\Pi} \bar{h}^{\alpha \beta} \bar{D}_{\alpha} e_{\beta n}
  \end{align}
  by {\eqref{eq:boundary derivative}} thus finished the proof.
\end{proof}

\end{document}